\documentclass[12pt]{amsart}
\usepackage{txfonts}
\usepackage{amscd,amssymb,mathabx,subfigure}
\usepackage[all]{xy}
\usepackage{graphicx,calrsfs}
\usepackage[colorlinks,plainpages,backref,urlcolor=blue]{hyperref}
\usepackage{tikz}
\usetikzlibrary{calc}
\usepackage{mdwlist}

\newtheorem{theorem}{Theorem}[section]

\newtheorem{lemma}[theorem]{Lemma}
\newtheorem{prop}[theorem]{Proposition}

\theoremstyle{definition}
\newtheorem{definition}[theorem]{Definition}
\newtheorem{example}[theorem]{Example}

\newtheorem*{ack}{Acknowledgments}

\newcommand{\Z}{\mathbb{Z}}
\newcommand{\Q}{\mathbb{Q}}

\newcommand{\C}{\mathbb{C}}
\newcommand{\FF}{\mathbb{F}}
\renewcommand{\L}{\mathbb{L}}
\newcommand{\TT}{\mathbb{T}}
\newcommand{\PP}{\mathbb{P}}

\renewcommand{\k}{\Bbbk}

\DeclareMathAlphabet{\pazocal}{OMS}{zplm}{m}{n}
\newcommand{\A}{{\pazocal{A}}}
\newcommand{\B}{{\pazocal{B}}}
\newcommand{\OO}{{\pazocal O}}

\newcommand{\LL}{{\pazocal L}}

\newcommand{\RR}{{\mathcal R}}
\newcommand{\VV}{{\mathcal V}}
\newcommand{\pP}{{\mathcal P}}

\newcommand{\h}{{\mathfrak{h}}}

\DeclareMathOperator{\rank}{rank}

\DeclareMathOperator{\im}{im}

\DeclareMathOperator{\Hom}{{Hom}}

\DeclareMathOperator{\Type}{Type}

\newcommand{\surj}{\twoheadrightarrow}

\def\dot{\mathchar"013A}
\newcommand{\hdot}{{\raise1pt\hbox to0.35em{\Huge $\dot$}}}

\definecolor{dkgreen}{RGB}{0,100,0}
\definecolor{dkbrown}{RGB}{139,69,19}

\begin{document}
\date{October 25, 2016}

\title[Modular equalities for complex reflection arrangements]{%
Modular equalities for complex reflection arrangements}

\author[A.D.~M\u acinic]{Daniela Anca~M\u acinic}
\address{Simion Stoilow Institute of Mathematics,
P.O. Box 1-764, RO-014700 Bucharest, Romania}
\email{Anca.Macinic@imar.ro}

\author[\c S.~Papadima]{\c Stefan Papadima$^1$}
\address{Simion Stoilow Institute of Mathematics,
P.O. Box 1-764,
RO-014700 Bucharest, Romania}
\email{Stefan.Papadima@imar.ro}
\thanks{$^1$Partially supported by the Romanian Ministry of
National Education, CNCS-UEFISCDI, grant PNII-ID-PCE-2012-4-0156}

\author[C.R.~Popescu]{Clement Radu Popescu$^2$}
\address{Simion Stoilow Institute of Mathematics, 
P.O. Box 1-764, RO-014700 Bucharest, Romania}
\email{Radu.Popescu@imar.ro}
\thanks{$^2$Supported by a grant of the Romanian
National Authority for Scientific Research, CNCS-UEFISCDI,
project number PN-II-RU-TE-2012-3-0492}



\begin{abstract}
We compute the combinatorial Aomoto--Betti numbers $\beta_p(\A)$
of a complex reflection arrangement. When $\A$ has rank at least $3$,
we find that $\beta_p(\A)\le 2$, for all primes $p$. Moreover, $\beta_p(\A)=0$
if $p>3$, and $\beta_2(\A)\ne 0$ if and only if $\A$ is the Hesse arrangement.
We deduce that the multiplicity $e_d(\A)$ of an order $d$ eigenvalue of the
monodromy action on the first rational homology of the Milnor fiber is equal to
the corresponding Aomoto--Betti number, when $d$ is prime. We give a uniform
combinatorial characterization of the property $e_d(\A)\ne 0$, for $2\le d\le 4$.
We completely describe the  monodromy action for full monomial arrangements of
rank $3$ and $4$. We relate $e_d(\A)$ and $\beta_p(\A)$ to multinets, on an
arbitrary arrangement.
\end{abstract}

\maketitle
\setcounter{tocdepth}{1}
\tableofcontents

\section{Introduction and statement of results}
\label{sec:intro}

\subsection{Milnor fibration and monodromy}
\label{ss11}

The complement $M$ of a degree $n$ complex hypersurface in $\C^l$,
$\{ f=0\}$, and the associated Milnor fibration, $f:M \to \C^{\times}$,
first analysed by Milnor in his seminal book \cite{M}, attracted a lot of attention
over the years. Multiplication by $\exp (\frac{2\pi \sqrt{-1}}{n})$ induces
the geometric monodromy action on the associated {\em Milnor fiber} $F=f^{-1}(1)$,
$h:F\to F$, and the {\em algebraic monodromy action},
$\{ h_i: H_i(F, \Q) \to H_i(F, \Q) \}$.

Computing $h_i$ is a major problem in the field, when $f$ has a non-isolated singularity  at $0$.
Even for the defining polynomial of a (central) complex hyperplane arrangement $\A$ in $\C^l$ and $i=1$, 
the answer is far from being clear.
This case was tackled in the recent literature by many authors,
who used a variety of tools; see for instance \cite{PS14} for a brief survey.
In this paper, we focus on {\em reflection arrangements}, associated to finite complex reflection groups.

It is well-known that every arrangement complement $M_\A$ has the homotopy type of a connected, 
finite CW-complex with torsion-free homology, whose first integral homology group, $H_1(M_\A, \Z)=\Z^n$, comes 
endowed with a natural basis, given by meridian loops around the hyperplanes.

It is also well-known that, for an arbitrary arrangement $\A$, $h_1$ induces a $\Q[\Z]$-module decomposition,
\begin{equation}
\label{eq:monintro}
H_1(F_\A, \Q)= \bigoplus_{d}(\Q[t] \slash \Phi_d(t))^{e_d(\A)},
\end{equation}
where $\Phi_d$ is the $d$-th cyclotomic polynomial, $e_d(\A)=0$ if $d \notdivides n$, and $e_1(\A)=n-1$.
See for instance \cite[(1.1)]{MP}. 

A pleasant feature of hyperplane arrangements is the rich combinatorial structure encoded by the associated 
intersection lattice, $\LL_\bullet(\A)$, whose elements are the intersections of hyperplanes from $\A$, ranked by 
codimension and ordered by inclusion. In this context, the open monodromy action problem takes the following 
more precise form: are the multiplicities $e_d(\A)$ combinatorial? If so, give a formula involving only $\LL_\bullet(\A)$.

\subsection{Characteristic and resonance varieties}
\label{ss12}

Our approach to decomposition \eqref{eq:monintro} is topological, based on two types of jump loci, 
associated to CW-complexes
having the properties recalled for $M_\A$, and the interplay between them.

 The complex characteristic variety $\VV^i_q(M)$, sitting inside the character torus \hfill
$\TT(M) \coloneqq \Hom (H_1(M, \Z), \C^\times)=(\C^\times)^n$
is the locus of those $\rho \in \TT(M)$ for which $\dim_\C H_i(M, \C_\rho) \geq q$, where $\C_\rho$
denotes the associated rank $1$ local system on $M$.

Note that $ \Hom (H_1(M, \Z), \C)=H^1(M, \C)=\C^n$, and denote by 
$\exp: H^1(M, \C) \surj \TT(M)$ the natural exponential map.
For an integer $d \geq 1$, let $\rho_d \in \TT(M)$ be the exponential 
of the diagonal cohomology class equal to
$\frac{2 \pi \sqrt{-1}}{d}$, with respect to the distinguished $\Z$-basis. 
When $M=M_\A$ is an arbitrary arrangement complement and $d>1$, it is well-known that

\begin{equation}
\label{eq:bmult}
e_d(\A)=\dim_\C H_1(M_\A, \C_{\rho_d}) \, .
\end{equation}

See for instance \cite[\S 2.3]{MP}, and also \cite[Theorem 2.5 and Corollary 6.4]{DS} for more general results
of this type.

The resonance variety $\RR_q^i (M, \k)$ over a field $\k$, sitting inside $H^1(M, \k)$, is the locus of those 
$\sigma \in H^1(M, \k)$ for which $\dim_\k H^i (H^\bullet (M, \k), \sigma *) \geq q$, where 
$\sigma *$ denotes the left multiplication
by $\sigma$ in the cohomology ring. When $\k$ is the prime field $\FF_p$, 
denote by $\sigma_p \in H^1(M, \FF_p)$ 
the diagonal cohomology class equal to $1$, with respect to the distinguished $\Z$-basis, 
and define the modulo $p$ {\em Aomoto-Betti number} by

\begin{equation}
\label{eq:defbeta}
\beta_p(M) \coloneqq \dim_{\FF_p} H^1(H^\bullet (M, \FF_p), \sigma_p *)
\end{equation}

\noindent When $M=M_\A$ is an arrangement complement, we will replace $M$ by $\A$ in the notation.

By a celebrated theorem of Orlik and Solomon \cite{OS}, the cohomology ring of $M_\A$ is combinatorial. 
More precisely, $\beta_p(\A)$ may be computed from $\LL_{\leq 2}(\A)$, as well as $\RR^1_q(\A, \k)$, for all $q$ and $\k$.

\subsection{Modular bounds}
\label{ss13}

It follows from Theorem 11.3 in \cite{PS-eq} that

\begin{equation}
\label{eq:bound}
e_{p^s}(\A) \leq \beta_p(\A), \text{for all}\; s \geq 1 \, ,
\end{equation}

when $\A$ is an arbitrary arrangement.

Actually, the above modular bound holds for all CW-complexes considered in \S\ref{ss12}, 
with the multiplicity replaced by the value from equality \eqref{eq:bmult}, and is in general strict, in the broader context. 
Our first main result says that the modular bound \eqref{eq:bound} becomes an equality, for reflection arrangements and $s=1$.

\begin{theorem}
\label{thm:a}
Let $\A$ be a complex reflection arrangement. Then $e_p(\A)=\beta_p(\A)$, for all primes $p$. 
In particular, $e_p(\A)$ is determined by $\LL_{\leq 2}(\A)$.
\end{theorem}

\subsection{Aomoto-Betti numbers for reflection arrangements}
\label{ss14}

Reflection arrangements have a distinguished history, going back as far as Jordan's work from 1878 on the symmetry group 
of the famous Hessian configuration. A related open problem is whether the {\em Hessian 
arrangement} is the only arrangement supporting a $4$-net. In Theorem \ref{thm:b}\eqref{b2} below, 
we solve this problem for reflection arrangements.

Finite complex reflection groups have been classified by Shephard and Todd \cite{ST} (see also \cite{C}, \cite{OT}).
Each such group $G$ gives rise to the complex reflection arrangement $\A(G)$, consisting of the fixed points of all reflections in $G$.
Among them, we have the monomial arrangements $\A(m,m,l)$ in $\C^l \; (l \geq 2)$, with defining polynomials
$\Pi_{1 \leq i< j \leq l}\ (z_i^m-z_j^m) \; (m \geq 1)$, and full monomial arrangements
$\A(m,1,l)$ in $\C^l \; (l \geq 2)$, defined by $z_1 \dots z_l \cdot \Pi_{1 \leq i< j \leq l}\ (z_i^m-z_j^m) \; (m \geq 2)$.
We may now state our second main result.

\begin{theorem}
\label{thm:b}
For a complex reflection arrangement $\A$ of rank at least $3$, the following hold.
\begin{enumerate}
\item \label{b1} If $p>3$, then $\beta_p(\A)=0$.
\item \label{b2} $\beta_2(\A)\neq 0
\Leftrightarrow \beta_2(\A)=2 \Leftrightarrow \A$ supports a $4$-net $\Leftrightarrow \A$ is the Hessian arrangement.
\item \label{b3} The only cases when  $\beta_3(\A)\neq 0$ are: $\A(m,1,3)$ with $m \equiv 1 \;(\text{mod} \;3)$, 
where $\beta_3=1$; $\A(m,m,3)$ with $m \geq 2$, where $\beta_3=1$ if $m \nequiv 0 \;(\text{mod}\; 3)$ 
and otherwise $\beta_3=2$; $\A(m,m,4)$, where $\beta_3=1$.
\item \label{b4} In particular, $\beta_p(\A) \leq 2$, for all primes $p$.
\end{enumerate}
\end{theorem}

We imposed the rank condition since the Aomoto-Betti numbers for arrangements of rank at most $2$ are known 
(see for instance \cite{MP}). Moreover, when $\A =\A(m,m,2)$ with $m \equiv 0 \;(\text{mod}\; p)$ it is easy
to see that $\beta_p(\A) = m-2$. Thus, 
the conclusion of Theorem \ref{thm:b}\eqref{b4}  no longer holds, for $m>4$.

The resonance varieties $\RR^1_q(\A, \C)$ are quite well-understood, due to work by  
Falk, Libgober, Marco-Buzun\'{a}riz and Yuzvinsky; see \cite{FY}, \cite{MB}. 
There are results in positive characteristic which show 
a different qualitative behaviour of resonance in this case; see for instance Falk  \cite{F07}. 
The complete picture over $\FF_p$ largely remains a mystery. 
Our Theorems \ref{thm:b} and \ref{thm:a}, together with recent vanishing results due to Dimca and Sticlaru
(\cite{Di}, \cite{DiS}), verify the strong modular conjecture from \cite{PS14}, for the important class 
of complex reflection arrangements.

\subsection{A combinatorial non-triviality test}
\label{ss15}

Dimca, Ibadula and M\u acinic asked in \cite{DIM} the following natural question: if $d>1$ and $e_d(\A) \neq 0$, 
does this imply that $\rho_d \in \exp\ \RR_1^1(\A, \C)$? A positive answer (for all $d$) would imply that 
the non-triviality of $h_1$ is combinatorial, since the converse implication is known, for all $d$.

\begin{theorem}
\label{thm:c}
If $\A$ is a complex reflection arrangement, then the above question has a positive answer, for $2 \leq d \leq 4$.
\end{theorem}

We derive Theorem \ref{thm:c} and Theorem \ref{thm:a} from Theorem \ref{thm:b} with 
the aid of  a general result (proved in Theorem \ref{thm:d})
that relates combinatorial structures on arrangements satisfying the key multinet axiom introduced by
Falk-Yuzvinsky in \cite{FY} to the algebraic monodromy action and the
Aomoto-Betti numbers of an arrangement. The tools from our paper also enable us to give a complete,
combinatorial description in Proposition \ref{prop:full} for the monodromy action on $H_1(F_\A, \Q)$, 
in the case of full monomial arrangements of rank $3$ and $4$.
Related results may be found in \cite{L} and \cite{PS14}. By Theorems \ref{thm:a}, \ref{thm:b} and \cite{Di, DiS},
this complete, combinatorial description holds for arbitrary complex reflection arrangements.

\section{Non-exceptional reflection arrangements}
\label{sec:fam}

We first compute the Aomoto-Betti numbers of monomial and full monomial arrangements.

\subsection{The classification} (cf. \cite{ST}, \cite{C}, \cite{OT})
\label{ss21}

A finite reflection group $G$ decomposes as a product of irreducible factors of the
same kind. At the level of arrangements, $\A(G\times G')$ is the product $\A(G)\times \A(G')$, and
the corresponding complements satisfy $M(G\times G') = M(G) \times M(G')$. (Whenever convenient, we will abbreviate
notation and replace  $\A(G)$ by $G$, when speaking about associated objects.)

The irreducible reflection arrangements of rank at least 3 comprise the monomial and full monomial arrangements,
$\A(m,m,l)\ (m\geq 2, l\geq 3\ {\rm or}\ m=1, l\geq 4)$ and $\A(m,1,l)\ (l\geq 3)$, plus the exceptional arrangements
$\A(G_{23}) - \A(G_{37})$. The Hessian arrangement is $\A(G_{25})$. See for instance \cite{OT} for the notations.

\subsection{Vanishing criteria}
\label{ss22}

Given an arbitrary arrangement $\A$ and an $r$--flat $X\in \LL_r(\A)$, 
set $\A_X \coloneq \{H\in\A | H\supseteq X\}$, and define
the multiplicity of $X$ to be equal to $|\A_X|$.

Using the distinguished $\Z$-basis, we identify an element $\eta\in H^1(M_\A,\FF_p)$ 
with the family $\{\eta_H\in \FF_p\}_{H\in\A}$.
We denote by $Z_p(\A)\subseteq H^1(M_\A,\FF_p)$ the kernel of $\sigma_p *$.

Definition \eqref{eq:defbeta} implies that $\beta_p(\A)=\dim_{\FF_p} Z_p(\A)-1.$
Our computations are based on the following well-known result (see e.g. \cite{PS14}).

\begin{lemma}
\label{lem:zeq}
An element $\eta$ belongs to $Z_p(\A)$ if and only if, for any $X\in\LL_2(\A)$,
\[
 \left\{
  \begin{array}{lccr}
\sum_{H \in\A_X}\eta_H = 0 & \text{if} & |\A_X|\equiv 0 & \text{(mod p)}\\
\eta_H = \eta_{H'}, \forall\ H, H'\in\A_X & \text{if} & |\A_X|\nequiv 0 & \text{(mod p)}
  \end{array}
 \right.
\]
\end{lemma}

Clearly, $\beta_p(\A)=0$ if and only if $\eta\in Z_p(\A)$ implies that $\eta$ is constant.
A first useful vanishing criterion is due to Yuzvinsky.

\begin{lemma}
\label{lem:yuzv}(\cite{Yuz01}) If $|\A|\nequiv 0$ (mod p), then $\beta_p(\A) = 0$.
\end{lemma}

A second convenient situation is the following.

\begin{lemma}
\label{lem:prod} Assume that $\A = \A_1\times\A_2$.
 \begin{enumerate}
 \item \label{p1} $H_1(M_{\A},\C_{\rho_d}) = 0$, for all $d>1$.
 \item \label{p2} $\beta_p(\A) = 0$, for all primes p.
 \end{enumerate}
\end{lemma}

\begin{proof}
Assuming the contrary, we infer from \cite[Proposition 13.1]{PS-bns} that
$\rho_d\in\VV_1^1(M_{\A_1}\times M_{\A_2}) = \{1\}\times\VV_1^1(M_{\A_2})\cup \VV^1_1(M_{\A_1})\times\{1\}$,
respectively 
$\sigma_p\in\RR_1^1(M_{\A_1}\times M_{\A_2},\FF_p) = \{0\}\times\RR_1^1(M_{\A_2},\FF_p)\cup 
\RR^1_1(M_{\A_1},\FF_p)\times\{0\}$,
in contradiction with the fact that all coordinates of $\rho_d$ (respectively $\sigma_p$) are different from 1 (respectively 0).
\end{proof}

By Lemma \ref{lem:prod}(\ref{p2}), we only need to compute $\beta_p(G)$ for an irreducible complex reflection group $G$.

To state the third vanishing criterion, we need to introduce certain simple graphs, associated to an arrangement $\A$ 
and an integer $k\geq 2$, with vertex set $\A$.
The edges of $\Gamma_k(\A)$ are defined by the condition $|\A_{H\cap H'}|\nequiv 0\ (\bmod\ k)$, for $k>2$, 
and by $|\A_{H\cap H'}|$ is either odd or equal to 2,
for $k=2$. The defining property for $\Gamma_{(k)}(\A)$ is $|\A_{H\cap H'}| = k$. Note that 
$\Gamma_{(2)}(\A)$ is a subgraph of $\Gamma_p(\A)$,
for all primes $p$. The equivalence relation on $\A$ associated to the edge paths of $\Gamma_k(\A)$, 
respectively $\Gamma_{(k)}(\A)$, will be denoted by
$\sim_k$, respectively $\sim_{(k)}$.

\begin{example}
\label{ex=braidgraphs}
Let $\A = \A(1,1,l)$ be the braid arrangement in $\C^l$, with hyperplanes labeled by the two-element subsets of
$\{ 1,\dots, l\}$. For $ij \ne st$, the multiplicity of the $2$--flat determined by the hyperplanes $z_i=z_j$ and $z_s=z_t$
is $2$ if $| \{ i,j,s,t \} |=4$ and $3$ if $| \{ i,j,s,t \} |=3$. Let $p$ be a prime. It follows from the above definition that 
$\Gamma_p(\A)$ is a complete graph (with full edge set) if $p\ne 3$. For $p=3$, $ij$ and $st$ are connected by an edge
of $\Gamma_3(\A)$ if and only if $| \{ i,j,s,t \} |=4$. We infer that the graph $\Gamma_3(\A)$ is discrete (with no edges)
when $l\le 3$, has $3$ connected components when $l=4$ (see the picture of the corresponding graph below), and is connected when $l\ge 5$.
\end{example}

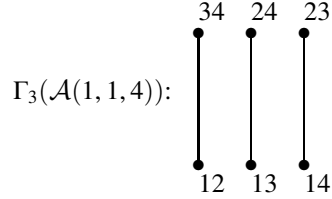
\begin{center}
\begin{picture}(360,70)
\put(90,35){{\footnotesize $\Gamma_3(\A(1,1,4))$:}}
   \put(160,10){\line(0,1){50}}         \put(160,10){\circle*{4}}
   
     \put(160,0){{\footnotesize $12$}}  \put(160,65){{\footnotesize$34$}}
    \put(160,60){\circle*{4}}
     \put(180,0){{\footnotesize $13$}}   \put(180,65){{\footnotesize $24$}}
      \put(180,10){\line(0,1){50}}       \put(180,10){\circle*{4}}
         \put(200,10){\line(0,1){50}}         \put(200,10){\circle*{4}}\put(180,60){\circle*{4}}           \put(200,60){\circle*{4}}
         \put(200,0){{\footnotesize $14$}}   \put(200,65){{\footnotesize $23$}}
\end{picture}
\end{center} 

We obtain the following immediate consequences of Lemma \ref{lem:zeq}.

\begin{lemma}
\label{lem:graphv} Each of the properties below implies that $\beta_p(\A)=0$.
 \begin{enumerate}
  \item\label{gv1} The graph $\Gamma_p(\A)$ is connected.
  \item\label{gv2} The graph $\Gamma_{(2)}(\A)$ is connected.
  \item\label{gv3} For all $X\in\LL_2(\A)$, $p\notdivides |\A_X|$.
 \end{enumerate}
\end{lemma}

\subsection{Intersection lattices}
\label{ss23} 

The Aomoto-Betti numbers for $\A(1,1,l)$ were computed in \cite{MP}.
They verify all statements from Theorems \ref{thm:b} and \ref{thm:a}. 
Hence, we may suppose from now on that $m\geq 2$.

We need to describe $\LL_{\leq 2}(\A)$. It will be convenient to label the various hyperplanes as follows.
Set $\omega= \exp (\frac{2\pi \sqrt{-1}}{m}),\ (H_i)\ z_i=0$ for all $1\leq i\leq l$,
and $(H_{ij^\alpha})\ z_i - \omega^\alpha z_j = 0$, for $1\leq i< j\leq l$ and $\alpha\in\Z/m\Z$.
We go on by listing the 2-flats (identified with the corresponding subarrangements $\A_X$).

\begin{description}
 \item[Case {\rm I}] $\A(m,1,l), l\geq 4:$
 \item[$\rm{I_a}$] \{$H_i, H_j, H_{ij^\alpha} (\alpha\in\Z/m\Z)$\}, with multiplicity $m+2$;
 \item[$\rm{I_b}$] \{$H_{ij^\alpha}, H_{jk^\beta}, H_{ik^{\alpha+\beta}}$\}, with multiplicity 3;
 \item[$\rm{I_c}$] \{$H_{ij^\alpha}, H_{kh^\beta}$\}, with multiplicity 2;
 \item[$\rm{I_d}$] \{$H_i, H_{jk^\alpha}$\}, with multiplicity 2.
 \item[Case {\rm II}] $\A(m,m,l), l\geq 4:$ types $\rm{I_b}$ and $\rm{I_c}$, plus \{$H_{ij^\alpha} (\alpha\in\Z/m\Z)$\},
 with multiplicity $m$.
 \item[Case {\rm III}] $\A(m,1,3):$ types $\rm{I_a}$, $\rm{I_b}$ and $\rm{I_d}$.
 \item[Case {\rm IV}] $\A(m,m,3):$ types $\rm{I_b}$ and $\rm{II}$.
\end{description}

\subsection{$\beta_p$-vanishing}
\label{ss24} 

We will use Lemma \ref{lem:graphv} to treat the cases when $\beta_p(\A) = 0$ 
in the non-exceptional families.
To simplify things, we suppress $H$ from notation and identify the hyperplanes 
with their labels, $i$ and $ij^\alpha$.

We claim that, for $\A = \A(m,1,l)$ with $l\geq 4, \Gamma_{(2)}(\A)$ is connected. Indeed, given $i<j$ we may find $h<k$
with $i,j,h,k$ distinct. Hence, $i\sim_{(2)} hk^0\sim_{(2)} j$ and $ij^\alpha\sim_{(2)} k$, which proves connectivity.
Similar arguments lead to the following conclusions. If $\A = \A(m,1,l)$ with $l=3$ and $p\neq 3$, then $\Gamma_p(\A)$
is connected; for $p = 3$ and $m\nequiv 1\ (\bmod\ 3), \Gamma_3(\A)$ is connected. 
The remaining full monomial Aomoto-Betti numbers,
$\beta_3(m,1,3)$ with $m\equiv 1\ (\bmod\ 3)$, will be computed later on.

If $\A = \A(m,m,l)$ with $l\geq 5$, then $\Gamma_{(2)}(\A)$ is connected. 
For $l = 3,4$ and $p\neq 3, \Gamma_p(\A)$
is connected. So, for monomial arrangements, only $\beta_3$ in ranks 3 and 4 remains to be calculated.

\subsection{The remaining non-exceptional cases}
\label{ss25}

A $\bmod\ 3$ cocycle $\eta\in Z_3(\A)$ is a family of elements of $\FF_3$, $\eta_k$ and $\eta_{ij^\alpha}$,
satisfying the equations from Lemma \ref{lem:zeq}, for any $X\in\LL_2(\A)$.

{\bf Case} $\A = \A(m,1,3)$ with $m\equiv 1 (\bmod\ 3)$.

The equations coming from 2-flats of type $\rm{I_d}$ say that
$\eta_{jk^\alpha} = \eta_i$, where $i$ is the third element of $\{1,2,3\}$.
The equations of type $\rm{I_b}$ become equivalent to $\eta_1+\eta_2+\eta_3 = 0$, while type $\rm{I_a}$
equations say that $\eta_i+\eta_j+m\eta_k = 0$, for all $i\neq j\neq k$. We infer that $\beta_3(m,1,3) = 1$,
as asserted in Theorem \ref{thm:b}.

{\bf Case} $\A = \A(m,m,4)$.

The equations of type $\rm{I_c}$ say that $\eta_{ij^\alpha} = \eta_{ij}$,
and $\eta_{34} = \eta_{12}, \eta_{24} = \eta_{13}, \eta_{23} = \eta_{14}$.
Type $\rm{I_b}$ equations reduce then to $\eta_{12}+\eta_{13}+\eta_{14} = 0$, 
while type II conditions follow from
$\eta_{ij^\alpha} = \eta_{ij}$. Again, $\beta_3(m,m,4) = 1$, as claimed.

{\bf Case} $\A = \A(m,m,3)$ with $m\nequiv 0 (\bmod\ 3)$.

The equations of type II say that $\eta_{ij^\alpha} = \eta_{ij}$, and the type $\rm{I_b}$ conditions
then reduce to $\eta_{12}+\eta_{23}+\eta_{13} = 0$. This shows that $\beta_3(m,m,3) = 1$, as asserted.

{\bf Case} $\A = \A(m,m,3)$ with $m = 3n$.

Set $\eta_{12^\alpha} = a_{\alpha}, \eta_{23^\alpha} = b_{\alpha}, \eta_{13^\alpha} = c_{\alpha}$.
With this notation, the equations of type $\rm{I_b}$ are equivalent to the system
\begin{equation}
\label{eq:neteq}
a_{\alpha}+b_{\beta}+c_{\alpha + \beta} = 0,\forall\ \alpha ,\beta\in\Z/m\Z,
\end{equation}
while the conditions of type II read
\begin{equation}
\label{eq:block}
\sum a_\alpha =\sum b_\beta =\sum c_\gamma = 0.
\end{equation}

We first solve the system \eqref{eq:neteq}, as follows. It implies that 
$a_\alpha + b_\beta = a_{\alpha '} + b_{\beta '}$,
if $\alpha + \beta = \alpha ' + \beta '$, in particular 
$a_\alpha - a_0 = b_\alpha  - b_0 \coloneq d_\alpha$, for all $\alpha$, and
$d_\alpha + d_\beta = d_{\alpha '} + d_{\beta '}$, if $\alpha + \beta = \alpha ' + \beta '$. 
We infer that $d_\alpha = \alpha d_1$,
for all $\alpha$. Hence $a_\alpha = a_0 + \alpha d_1, b_\alpha = b_0 + \alpha d_1$ 
and $c_\alpha = -a_0 - b_0 - \alpha d_1$,
which solves the system \eqref{eq:neteq}. In particular, its solution space is 3-dimensional.

Finally, it is an easy matter to check that \eqref{eq:neteq} $\Rightarrow$ \eqref{eq:block}, since $m = 3n$.
Therefore, $\beta_3(3n,3n,3) = 2$, as asserted.

This proves Theorem \ref{thm:b} for non-exceptional reflection arrangements.

\section{Exceptional reflection arrangements}
\label{sec:exc}

We finish the proof of Theorem \ref{thm:b}, by computing the 
Aomoto-Betti numbers of the exceptional complex reflection
arrangements of rank at least 3, $G_{23} - G_{37}$.

\subsection{The groups $G_{31},\ G_{32},\ G_{33}$}
\label{ss31} 

{\bf Case} $\A = \A(G_{31})$. The hyperplanes of $\A$ live in $\C^4$. Their defining equations are as follows
(see \cite{HR}). Set $\omega= \exp (\frac{2\pi \sqrt{-1}}{4})$. 
The hyperplanes of $\A$ are: 
\begin{itemize}
\item $(H_i)\ z_i = 0\ (1\leq i\leq 4)$;
\item $(H_{ij^\beta})\ z_i -\omega^\beta z_j = 0\ (1\leq i< j\leq 4,\ \beta\in\Z/4\Z)$;
\item $(H_{\overline\alpha})\ z_1 + \sum_{2\leq i\leq 4} \omega^{\alpha_i}z_i = 0$
$(\overline\alpha = (\alpha_2,\alpha_3,\alpha_4)\in(\Z/4\Z)^3, \alpha_2 + \alpha_3 + \alpha_4\equiv 0\ (\bmod\ 2))$.
\end{itemize}

By Lemma \ref{lem:graphv}(\ref{gv2}), it is enough to show that $\Gamma_{(2)}(G_{31})$ is connected.
This can be seen as follows. Clearly, the 2-flat $H_k\cap H_{ij^\beta}$ has multiplicity 2, when $i\neq j\neq k$.
This implies that $i\sim_{(2)} j\sim_{(2)} kh^{\beta}$, for all $1\leq i<j\leq 4,\ 1\leq k<h\leq 4$ and
$\beta\in\Z/4\Z$. Given any $H_{\overline\alpha}$, it is not hard to see that the multiplicity of
$H_{\overline\alpha}\cap H_{12^{\alpha_2}}$ is 2. This proves connectivity, as claimed.

{\bf Case} $\A = \A(G_{32})$. Set $\omega= \exp (\frac{2\pi \sqrt{-1}}{3})$. The arrangement $\A$
consists of the following hyperplanes in $\C^4$ (see \cite{Sz}): 
\begin{itemize}
\item $(H_i)\ z_i = 0\ (1\leq i\leq 4)$;
\item $(H_{1^{\alpha\beta}})\ z_2 + \omega^\alpha z_3 + \omega^\beta z_4 = 0$;
\item $(H_{2^{\alpha\beta}})\ z_1 + \omega^\alpha z_3 - \omega^\beta z_4 = 0$;
\item $(H_{3^{\alpha\beta}})\ z_1 - \omega^\alpha z_2 + \omega^\beta z_4 = 0$;
\item $(H_{4^{\alpha\beta}})\ z_1 + \omega^\alpha z_2 - \omega^\beta z_3 = 0$ ($\alpha,\beta\in\Z/3\Z$).
\end{itemize}
Clearly, the 2-flats $H_i\cap H_j\ (i\neq j)$ and $H_i\cap H_{i^{\alpha\beta}}$ have multiplicity 2.
This shows that $\Gamma_{(2)}(G_{32})$ is connected and we are done.

{\bf Case} $\A = \A(G_{33})$. Here, $\omega= \exp (\frac{2\pi \sqrt{-1}}{3})$ and the hyperplanes (in $\C^6$)
are as follows (see \cite{C,R}): 
\begin{itemize}
\item $(H_{ij^\beta})\ z_i - \omega^\beta z_j = 0\ (1\leq i<j\leq 4,\ \beta\in\Z/3\Z)$;
\item $(H_{\overline\alpha})\ \sum_{1\leq i\leq 4} \omega^{\alpha_i}z_i + z_5 + z_6 = 0$
$\ (\overline\alpha = (\alpha_1,\alpha_2,\alpha_3,\alpha_4)\in(\Z/3\Z)^4,\ \sum\alpha_i = 0)$.
\end{itemize}
Plainly, $ij^{\beta}\sim_{(2)}kh^0\sim_{(2)}ij^{\beta '}$, for all $\beta, \beta '$ (where $i\neq j\neq k\neq h$).
Like in the  Case $G = G_{31}$, it can be checked that
$H_{\overline\alpha}\cap H_{ij^\beta}$ has multiplicity 2, if $\beta = \alpha_j - \alpha_i$.
We infer that $\Gamma_{(2)}(G_{33})$ is connected, and we are done.

\subsection{More vanishing criteria}
\label{ss32} 

We will no longer need defining equations to settle the remaining cases. 
We will use instead a couple of new vanishing arguments.

For the beginning, let us recall from \cite[pp.~224-225]{OT} the following very useful properties of reflection
groups and arrangements, derived from a key result of Steinberg \cite{S}. For any complex reflection group $G$ and any
$X\in\LL_r(\A(G))$, the fixer subgroup $G_X \coloneq \{g\in G\ |\ gx = x,\ \forall x\in X\}$
is a reflection group, and $\A(G)_X = \A(G_X)$ is again a reflection arrangement, of rank $r$.
By the construction of $\A(G)$, the group $G$ acts on the arrangement $\A(G)$, hence on the intersection lattice
$\LL_\bullet(G)$. Let us denote by $\OO_X$ the $G$-orbit of $X\in\LL(G)$. Let $\Type (X)$ be the isomorphism type
of the reflection group $G_X$. It follows from \cite[Lemma 6.88]{OT} that the type is constant on each orbit $\OO_X$.
Moreover, Table C from \cite{OT} gives $|\A(G)|$ and the orbit partition of $\LL_\bullet(G)$ for all exceptional groups,
in terms of types of orbits.

This leads to the quick computation of the sets
\[
\pP(G) := \{ p\ \text{prime}\ |\ \exists\ X\in\LL_2(G)\ \text{such that}\ |\A(G_X)|\equiv 0\ (\bmod\ p)\}
\]

In particular, $\pP(G)\subseteq \{2,3,5\}$, for every exceptional arrangement of rank at least 3.
We infer from Lemma \ref{lem:graphv}(\ref{gv3}) that $\beta_p(G) = 0$, if $p>5$.
Hence, we may suppose from now on that $p\leq 5$.

For an arbitrary arrangement $\A, H\in\A$ and a prime $p$, we define
\begin{equation}
\label{eq:multp}
m_p(H) = 1 + \sum(|\A_X| - 1),
\end{equation}
where the sum is taken over those $X\in\LL_2(\A)$ such that $X\subseteq H$
and $|\A_X|\nequiv 0\ (\bmod\ p)$.

The numbers from \eqref{eq:multp} may be extracted from Table C in \cite{OT},
for exceptional reflection arrangements of rank at least 3. This is based on the following fact, valid for an
arbitrary reflection group $G$. For $X, Y\in\LL(G)$, let $u(X,Y)$ be the number of $Z\in\LL(G)$ such that
$Z\in\OO_Y$ and $Z\subseteq X$. Clearly, this number depends only on $\OO_X$ and $\OO_Y$.
The values $u(H,X)$ may be found in Table C, for all orbit types corresponding to $H\in\LL_1(G)$
and $X\in\LL_2(G)$.

\begin{lemma}
\label{lem:frac} For an arrangement $\A$ and a prime $p$, the following hold.
 \begin{enumerate}
 \item\label{f12} If $m_p(H)>\frac{|\A|}{2}$ for all $H\in\A$, then $\beta_p(\A) = 0.$
 \item\label{f13} If $m_p(H)>\frac{|\A|}{3}$ for all $H\in\A$ and $\A$ has no rank 2 flats of
 multiplicity $p\cdot r$ with $r>1$, then $\beta_p(\A) = 0$.
 \end{enumerate}
\end{lemma}

\begin{proof} 
If $\beta_p(\A)\neq 0$, there is a non-constant function $\eta\in\FF_p^\A$
satisfying all equations from Lemma \ref{lem:zeq}. Fix $H\in\A$ and set $\eta_H = \alpha$.
We claim that $|\{\eta = \alpha\}|\geq m_p(H)$. Indeed, if $X\in\LL_2(\A)$ is contained in $H$
and $|\A_X|\nequiv 0\ (\bmod\ p)$, then $\eta$ must have the constant value $\alpha$ on $\A_X$.
An easy count of all these hyperplanes gives the claimed inequality.

In Part (\ref{f12}), this implies that $\eta$ must be constant, a contradiction. In Part (\ref{f13}),
we infer that $\eta$ has only two distinct values. By adding constant functions 
and multiplying by non-zero elements in
$\FF_p$, we may assume that these values are $\eta_H = 0$ and $\eta_{H '}= 1$.
By Lemma \ref{lem:zeq}, the flat $X = H\cap H'$ has multiplicity $p\cdot r$, 
imposing the condition $\sum_{K\in\A_X} \eta_K = 0$.
Since necessarily $r = 1$, we arrive again at a contradiction.
\end{proof}

\subsection{Induction on rank}
\label{ss33}

We start with a couple of general considerations. A subarrangement $\B\subseteq\A$ is called
{\em line-closed in} $\A$ (see the first definition from \cite[Definition 1.1]{F02}) 
if $\B_X = \A_X$, for all $X\in\LL_2(\B)$. This property implies that the
restriction map, $\FF_p^\A\longrightarrow\FF_p^\B$, sends $Z_p(\A)$ into $Z_p(\B)$.
Clearly, $\A_Y$ is line-closed in $\A$, for any $Y\in\LL (\A)$.

\begin{lemma}
\label{lem:indr} Let G be a complex reflection group. Assume that $2\leq r<\rank(\A(G))$ and
$\beta_p(\A(G')) = 0$, for all irreducible groups $G'\in \Type(\LL_r(\A(G))).$ Then $\beta_p(\A(G)) = 0$.
\end{lemma}

\begin{proof} 
Assuming the contrary, there exist $\eta\in Z_p(G)$ and $H_1,\ H_2\in \A(G)$ such that
$\eta_{H_1}\neq\eta_{H_2}$. From our assumption on $r$, we may find $H_3,\dots ,\ H_r\in\A(G)$ such that
$X = H_1\cap H_2\cdots\cap H_r\in\LL_r (G)$. Set $\B = \A(G)_X = \A(G_X)$. We deduce that $\beta_p(G_X)\neq 0$.
If $G_X$ is reducible, this contradicts Lemma \ref{lem:prod}. Otherwise, our second assumption is violated.
\end{proof}

\subsection{The rank 3 case}
\label{ss34}

The rank 3 exceptional groups are $H_3, G_{24}, G_{25}, G_{26}$ and $G_{27}$. Table C from \cite{OT} provides
the following information on each group:

\begin{itemize}
 \item $\pP(G) = \{2,3,5\};\ \{2,3\};\ \{2\};\ \{2,5\};\ \{2,3,5\};$
 \item $|\A(G)|$ = 15; 21; 12; 21; 45.
\end{itemize}
By Lemma \ref{lem:yuzv} and Lemma \ref{lem:graphv}(\ref{gv3}), the only 
primes $p$ which might give $\beta_p(G)\neq 0$ are as follows:
$$ 3,5\ (H_3);\ 3\ (G_{24});\ 2\ (G_{25});\ -\ (G_{26});\ 3,5\ (G_{27}).$$ 

Using Lemma \ref{lem:frac}(\ref{f12}), we obtain
$\beta_3(H_3) = \beta_3(G_{24}) = \beta_3(G_{27}) = \beta_5(G_{27}) =0$, 
and Lemma \ref{lem:frac}(\ref{f13}) gives $\beta_5(H_3) = 0$.
Finally, $\beta_2(G_{25}) = 2$, cf. \cite{PS14}. Thus, Theorem \ref{thm:b} is proved in this case. Indeed, the Hessian arrangement
$\A(G_{25})$ supports a 4-net (see e.g. \cite{Yuz04}). This implies that $\beta_2\neq 0$, by \cite{PS14}; the other implications
from Theorem \ref{thm:b}(\ref{b2}) are obvious.

Applying Lemma \ref{lem:indr} for $r = 3$ and $p = 5$, we also infer that Theorem \ref{thm:b}(\ref{b1}) holds
for all complex reflection arrangements of rank at least $3$. Thus, we only need to show that $\beta_2(G) = \beta_3(G) = 0$,
when $G$ is exceptional of rank at least 4, in order to complete the proof of Theorem \ref{thm:b}.

\subsection{The remaining cases}
\label{ss35}

The only rank 5 exceptional arrangement is $G_{33}$, for which we know from \S\ref{ss31} that all $\beta_p$ vanish.
By the computations from Section \ref{sec:fam}, the same thing happens for non-exceptional irreducible arrangements of rank 5.
Lemma \ref{lem:indr}, applied for $r = 5$, guarantees then that we may reduce our proof to the rank 4 case. Here, the list is
$G = F_4, G_{29}, H_4, G_{31}, G_{32}$, and the last two groups were treated in \S\ref{ss31}.

{\bf Case} $G = F_4$. The irreducible rank 3 types are listed in Table C from \cite{OT}: $B_3$ and $C_3$; 
in both cases, the arrangement
$\A(G')$ is $\A(2,1,3)$, for which all $\beta_p$ vanish, cf. Section \ref{sec:fam}. We may conclude by resorting to
Lemma \ref{lem:indr} for $r = 3$.

{\bf Case} $G = G_{29}$. The list of irreducible rank 3 types is: $G' = A_3, B_3, G(4,4,3)$. Taking $r = 3$ and $p = 2$
in Lemma \ref{lem:indr}, we deduce from Section \ref{sec:fam} that $\beta_2(G_{29}) = 0$. Since $|\A(G_{29})| = 40$,
$\beta_3(G_{29}) = 0$, by Lemma \ref{lem:yuzv}.

{\bf Case} $G = H_4$. The irreducible types of $\LL_3(G)$ are: $G' = A_3, H_3$. Again 
by Lemma \ref{lem:indr} and previous computations,
$\beta_2(H_4) = 0$. Finally, $\beta_3(H_4) = 0$, as follows from Lemma \ref{lem:frac}(\ref{f12}). 

The proof of Theorem \ref{thm:b} is complete.

\section{Multinets and jump loci}
\label{sec:multi}

In this section we prove Theorems \ref{thm:a} and \ref{thm:c}. 
Along the way, we establish a useful general result that relates
combinatorial structures on arrangements satisfying the main multinet axiom to the algebraic monodromy 
of its Milnor fibration and its Aomoto-Betti numbers.

\subsection{Multinets and weighted partitions}
\label{ss41}

The work of Falk and Yuzvinsky from \cite{FY} gives, among other things, 
a description of the resonance variety of an arrangement $\A$,
$\RR^1_1(\A, \C)$, in terms of {\it multinets} on the associated matroid.
A $k-$multinet on $\LL_{\leq 2}(\A)$ is a partition $\Pi$ with $k\geq 3$
non-empty blocks, $\A= \bigsqcup _{\alpha \in [k]} \A_{\alpha}$,
together with a function, $\mathfrak{m}: \A \rightarrow \Z_{> 0}$,
satisfying certain axioms. The most important is the following:

For any $H \in \A_{\alpha}$ and $H' \in \A_{\beta}$ with $\alpha \neq \beta$, and every $\gamma \in [k]$,
\begin{equation}
\label{eq:netax}
n_X  \coloneqq \sum_{K \in \A_{\gamma} \cap \A_X} \mathfrak{m}_K
\end{equation}

is independent of $\gamma$, where $X= H \cap H' \in \LL_2(\A)$.

Let $\Pi$ be a partition of $\A$ with $k\geq 3$ non-empty blocks, as above. Let $\mathfrak{m}: \A \rightarrow \Z$
be an assignment of arbitrary integer weights to the hyperplanes of $\A$. 

\begin{definition}
\label{def=cpart}
The pair $\mathcal{N}=(\Pi, \mathfrak{m})$ is a {\em weighted $k$-partition} if axiom \eqref{eq:netax}
is satisfied. We will say that $\mathcal{N}$ is {\em $h$-reduced} $(h \geq 1)$ if $\mathfrak{m}_K \equiv 1 (\bmod\; h)$,
for all $K\in \A$. 
\end{definition}

The underlying partition of a $k$-multinet, together with its positive weight function, is a weighted $k$-partition.
Moreover, the usual notion of reduced multinet corresponds to $h=1$.

 We will need a result from \cite{PS14} related to axiom \eqref{eq:netax}.
 To recollect it, we start with a few notations.
 Set $H_{\A} \coloneqq H_1(M_{\A}, \Z)$ and denote by $\{a_H\}_{H \in \A}$ the distinguished $\Z$-basis. Let $S$
be  $\C \PP^1 \setminus \{k \; \text{points}\}$ and set
$H_S \coloneqq H_1 (S, \Z)=\Z-\text{span}\langle c_\alpha | \alpha \in [k]\rangle \slash \sum _{\alpha \in [k]} c_\alpha$,
where $c_{\alpha}$ is the class of a small loop in $S$ around the point $\alpha$.

Let $\cup_{\A}: \bigwedge^2 H^1(M_{\A}, \Z) \rightarrow H^2(M_{\A}, \Z)$ be the cup product.
Recalling from \cite{OS} that $H^\bullet (M_{\A}, \Z)$ has no torsion,
we denote by $\nabla_{\A}: H_2(M_{\A}, \Z) \rightarrow \bigwedge^2 H_{\A}$ the $\Z$-dual comultiplication map.

The next result improves Theorem 2.4 from \cite{FY}, in several ways. The hypothesis of Proposition \ref{prop:multilie} 
is reduced to the key axiom $(iii)$ from \cite[Definition 2.1]{FY}. The weight function $\mathfrak{m}$ may take arbitrary 
integer values, while in \cite{FY} positivity plays a crucial role. The conclusion in \cite{FY} is that 
$\im(\phi \otimes \k)^* \subseteq H^1(M_{\A}, \k)$ is an isotropic subspace, for a characteristic $0$ field $\k$
(in positive characteristic, an additional condition on $\mathfrak{m}$ is needed in \cite{FY}), while  
Proposition \ref{prop:multilie} gives the conclusion over $\Z$. For the reader's convenience, we include the proof.

\begin{prop}[\cite{PS14}]
\label{prop:multilie} 
Let $\mathcal{N}=(\Pi, \mathfrak{m})$ be a weighted $k$-partition. 
Then $\bigwedge^2  \phi  \circ \nabla _{\A}=0$,
where $\phi : H_{\A} \rightarrow H_S$ sends $a_H$ to $ \mathfrak{m}_H c_{\alpha}$, for $H \in \A_{\alpha}$.
Therefore, $\cup_{\A} \circ \bigwedge^2 \phi ^*=0$, by taking $\Z$-duals.
\end{prop}

\begin{proof}
Let $\h (\A)$ be the $\Z$-form of the holonomy Lie algebra of $\A$, appearing in Proposition 5.2 from \cite{PS14}.
By definition, this is the graded $\Z$-Lie algebra quotient of the free $\Z$-Lie algebra generated by $H_{\A}$,
$\L^{\hdot}(H_{\A})$, graded by bracket length, by the graded Lie ideal generated by 
$\im (\nabla _{\A}) \subseteq \bigwedge^2 H_{\A}$, where $\bigwedge^2 H_{\A}$ is identified with $\L^{2}(H_{\A})$
via the Lie bracket. Let $\L^{\hdot}(\phi) : \L^{\hdot}(H_{\A}) \to \L^{\hdot}(H_{S})$ be the graded $\Z$-Lie algebra 
map extending $\phi$. Our claim says that $\L^{\hdot}(\phi)$ factors through $\h (\A)$.

To check this, we recall from \cite[(30)]{PS14} that the defining Lie relations of $\h (\A)$ are:
\[
\sum_{K\in \A_X} [a_K, a_L] \, , \quad \text{for} \quad X\in \LL_2(\A) \quad \text{and} \quad L\in \A_X \, .
\]
Thus, we have to show that $[\sum_{K\in \A_X} \phi (a_K), \phi (a_L)]=0 \in \L^{2}(H_{S})$. 

There are two cases to consider. When $X$ is mono-coloured, i.e., $\A_X \subseteq \A_{\alpha}$
for some $\alpha\in [k]$, by construction $\phi (a_K)\in \Z \cdot c_{\alpha}$, for all $K\in \A_X$,
and we are done. Otherwise, $X= H \cap H'$, with  $H \in \A_{\alpha}$, $H' \in \A_{\beta}$ and $\alpha \neq \beta$.
Again by construction, 
$\sum_{K\in \A_X} \phi (a_K)= \sum_{\gamma \in [k]} (\sum_{K \in \A_{\gamma} \cap \A_X} \mathfrak{m}_K) c_{\gamma}$, 
which equals $n_X(\sum_{\gamma \in [k]} c_{\gamma})$, by axiom \eqref{eq:netax}. This implies that 
$\sum_{K\in \A_X} \phi (a_K)=0 \in \L^{1}(H_{S})$, by the definition of $H_{S}$, which completes the proof.
\end{proof}

\subsection{Relating weighted partitions to jump loci}
\label{ss42}

We are now ready to state our result, keeping the previous notation.

\begin{theorem}
\label{thm:d}
Assume that $\mathcal{N}=(\Pi, \mathfrak{m})$ is a $k$-reduced weighted $k$-partition.
Then the following hold, for all divisors $p, d$ of $k$ with $p$ prime and $d>1$.
\begin{enumerate}
\item \label{d1} $\rho_d(\A) \in \exp\ \RR_1^1(\A, \C)$, in particular $e_d(\A)>0$.
\item \label{d2} $\beta_p(\A) \neq 0$.
\end{enumerate}
\end{theorem}

\begin{proof}
Part \eqref{d2}. Since the weighted partition is $k$-reduced and $p|k, \; \phi \otimes \k$ is surjective, by construction, 
for $\k=\FF_p$ and $\C$. It follows from Proposition \ref{prop:multilie} that $\im(\phi \otimes \k)^* \subseteq H^1(M_{\A}, \k)$ 
is a $(k-1)$-dimensional subspace, isotropic with respect to the cup product.
The linear map sending each $c_{\alpha}$ to $1 \in \FF_p$ defines an element of $H^1(S, \FF_p)$, 
denoted $\sigma_p(S)$. Clearly, $\phi^* (\sigma_p(S))=\sigma_p(\A)$, since $\mathcal{N}$ is in particular  $p$-reduced. 
By definition \eqref{eq:defbeta}, $\beta_p(\A) \neq 0$ as claimed, since $k \geq 3$.

Part \eqref{d1}. Note that $\phi: H_{\A} \rightarrow H_S$ induces homomorphisms $\phi^*:
\mathbb{T}(S) \rightarrow \mathbb{T} (M_{\A})$ and $\phi ^* :  H^1(S, \C) \hookrightarrow H^1(M_{\A}, \C)$, 
compatible with the surjective exponential maps of $S$ and $M_{\A}$. 
The map sending each $c_{\alpha}$ to $\exp (\frac{2\pi \sqrt{-1}}{k})$ 
defines an element of the character torus, $\rho_k(S) \in \mathbb{T}(S)$. Plainly, $\phi^* (\rho_k(S))=\rho_k(\A)$, 
since $\mathcal{N}$ is $k$-reduced. We also have $(\rho_k)^{ k \slash d}= \rho_d$, for both $S$ and $\A$, since 
we assumed that $d$ divides $k$. Hence, 
$\phi^*(\rho_d(S))=\rho_d(\A) \in \exp \; \phi^*(H^1(S, \C)) \subseteq \exp \; \RR^1_1(\A, \C)$, 
where the last inclusion follows from the argument in Part \eqref{d2}. 
Indeed, we know that $\phi^*(H^1(S, \C))$ is an isotropic subspace in $H^1(M_\A, \C)$, of dimension at least $2$, 
and we may simply use the definition of $\RR^1_1$. The conclusion $e_d(\A) > 0$ is a direct consequence 
of equality \eqref{eq:bmult}, since it is well-known that $\exp \; \RR_1^1(\A, \C) \subseteq \mathcal{V}^1_1 (M_\A)$; 
see e.g. \cite[Theorem D]{DP} for a more general result.
\end{proof}

\subsection{Reduced weighted partitions on complex reflection arrangements}
\label{ss43}

We begin the proof of Theorem \ref{thm:c}. Let $\A$ be a complex reflection arrangement and assume $e_d(\A)>0$,
with $2\le d\le 4$. We will show that $\rho_d(\A) \in \exp\ \RR_1^1(\A, \C)$ with the aid of Theorem \ref{thm:d}\eqref{d1},
which requires the existence of a certain weighted partition on $\A$. 

An easy preliminary remark is that the question  from \cite{DIM} 
always has a positive answer, for any arrangement $\A$ of rank at most $2$. To see this, note first that the assumption 
$e_d(\A) \neq 0$ is equivalent to $\rho_d \in \mathcal{V}^1_1(M_\A)$, by equality \eqref{eq:bmult}. When rank$(\A) \leq 2$, 
it is known that $\mathcal{V}^1_1(M_{\A})=\exp \; \RR_1^1(\A, \C)$, so the conclusion follows trivially. Consequently, 
we may also suppose that the rank is at least $3$. On the other hand, $e_d(\A)>0$ and $d=p^s$ together imply, 
via the modular bound \eqref{eq:bound}, that $\beta_p(\A)>0$. Therefore, $\A$ must be either the Hessian arrangement, or one 
of the arrangements from Theorem \ref{thm:b}(\ref{b3}). To apply Theorem \ref{thm:d}(\ref{d1}), we need to  
describe suitable weighted partitions on these arrangements, in each case.

The Hessian arrangement supports a reduced $4$-multinet  (actually, a $4$-net). The monomial arrangement $\A(m,m,3)$ 
has a reduced $3$-multinet (in fact, a $3$-net), as noted in \cite{FY}.

A (non-reduced) $3$-multinet on the full monomial arrangement $\A(m,1,3)$ was constructed in \cite{FY}. It is immediate 
to check that the weighted partition associated to this multinet is $3$-reduced, when $m \equiv 1 (\text{mod} \; 3)$.

The last case is $\A=\A(m,m,4)$, with hyperplanes $(H_{ij^\mu}) \; z_i - \omega^\mu z_j=0$, where 
$1 \leq i<j \leq 4, \; \mu \in \Z \slash m\Z$ and $\omega=\exp (\frac{2 \pi \sqrt{-1}}{m})$. We define a partition $\Pi$ 
with three blocks, $\{H_{ij^\mu}, H_{kh ^\nu}| \; \mu, \nu \in \Z \slash m\Z\}$, and set $\mathfrak{m} \equiv 1$ on $\A$. 
It is straightforward to verify axiom \eqref{eq:netax} by using the description of $2$-flats given in  \S\ref{ss23}. 
(Actually, this is a $3$-net on $\LL_{\leq 2}(\A)$.)

\subsection{Proof of Theorem \ref{thm:c} completed}
\label{ss44}

In case $\A=\A(G_{25})$, $d$ must be $2$ or $4$. We may take $k=4$ in Theorem \ref{thm:d} to obtain the desired conclusion.
The remaining cases, described in Theorem \ref{thm:b}(\ref{b3}), lead to $d=3$. Taking $k=3$, we conclude as before.
\hfill $\Box$

\subsection{Proof of Theorem \ref{thm:a}}
\label{ss45}

We may suppose that $\rank (\A) \geq 3$, since otherwise the conclusion is known (see \cite[p. 773]{MP}). 
By  the modular bound \eqref{eq:bound} and
Theorem \ref{thm:b}, $e_p(\A)=\beta_p (\A)$, when $\A$ is not $\A(G_{25})$ or one of 
the arrangements listed in Theorem \ref{thm:b}(\ref{b3}).
Moreover, we have to verify the conclusion only for $p=2$ (in case $\A(G_{25})$) or for $p=3$ (in the remaining cases).

 The equality $e_2(G_{25})=\beta_2(G_{25})=2$ is well-known (see e.g. \cite{PS14}). When $\A$ is not $\A(m,m,3)$ 
with $m \equiv 0\ (\bmod \; 3)$, we know that $\beta_3(\A)=1$. In these cases, 
we may use the $3$-reduced weighted $3$-partitions 
from \S\ref{ss43}, for $d=k=3$, exactly as in \S\ref{ss44}, to obtain that $e_3(\A)>0$. Now we are done, 
since the modular bound \eqref{eq:bound} implies that $e_3(\A) \leq 1$.

The last case, $\A=\A(m,m,3)$ with $m=3n$ and $p=3$, when $\beta_3(\A)=2$, requires a more careful treatment.

The (relabeled) hyperplanes of $\A$ are 
$\{H_{12^\alpha},\; H_{23^\beta}, \; H_{13^{-\gamma}}\ |\ \alpha, \beta, \gamma \in \Z \slash m\Z\}$. We recall 
from \S\ref{ss23} the two types of $2$-flats: $\{H_{ij^\alpha}\ | \; \alpha \in  \Z \slash m\Z\}$ and 
$\{H_{12^\alpha},\; H_{23^\beta}, \; H_{13^{-\gamma}}\}$ with $\alpha + \beta + \gamma=0$.

 We have to show that $e_3(\A) \geq 2$, in order to finish the proof. To this end, we need two reduced 
weighted $3$-partitions on $\A$.
 The first one, $\mathcal{N}$, is constructed in \cite{FY}. The blocks of the partition $\Pi$ are given by
 $\{H_{ij^\alpha}\ |\;  \alpha \in  \Z \slash m\Z\}_{1 \leq i< j \leq 3}$.
 The blocks of the second one, $\mathcal{N}'$, are defined by 
$\{H_{ij^\alpha}\ | \; 1 \leq i< j \leq 3, \alpha \equiv \tau (\bmod \;3)\}_{\tau \in \FF_{3}}$, 
with $\alpha$ replaced by $- \alpha$ when $\{i,j \}=\{1,3 \}$.

 Let us check axiom \eqref{eq:netax} for $\mathcal{N}'$. The $2$-flats $X$ appearing in axiom \eqref{eq:netax} clearly
 coincide with those $\A_X$ that contain two hyperplanes with different colours with respect to $\Pi '$.
 For $X=\{H_{ij^\alpha}| \;\alpha \in  \Z \slash m\Z \}$ we find that $n_X=n$.
For $X=\{H_{12^\alpha}, \; H_{23^\beta}, \; H_{13^{-\gamma}}\}$ with $\alpha + \beta + \gamma=0$,
the condition on colours translates to  $\alpha \nequiv \beta \nequiv \gamma (\bmod\ 3)$,
and implies that $n_X=1$. Hence $\Pi'$ defines a reduced weighted $3$-partition $\mathcal{N}'$.

Consider the two (surjective)  homomorphisms from Proposition \ref{prop:multilie}, $\phi, \phi' : H_\A \rightarrow H_S$. 
Clearly, $\phi^* \mathbb{T}(S)= \exp \; \phi^* H^1(S, \C)$
is a positive dimensional subtorus of $\mathbb{T} (M_\A)$, and similarly for $\phi'$. Moreover, 
$\phi^* \mathbb{T}(S) \subseteq \mathcal{V}_1^1(M_\A)$, since  $\phi^* H^1(S, \C) \subseteq H^1(M_\A, \C)$ is 
isotropic of dimension $2$, hence contained in $\RR_1^1(\A, \C)$, and likewise for $\phi'$. Therefore, 
we may find two irreducible components of $\mathcal{V}^1_1(M_\A)$, $W$ and $W'$, such that 
$\phi^*\mathbb{T}(S) \subseteq W$ and $\phi '^*\mathbb{T}(S) \subseteq W'$. On the other hand, 
$\rho_3 (\A) \in W \cap W'$, by the argument from the proof of Theorem \ref{thm:d}(\ref{d1}).

In this situation, it follows from a result of Artal Bartolo, Cogolludo and Matei
\cite[Proposition 6.9]{ACM} that $\rho_3(\A) \in \mathcal{V}_2^1(M_\A)$, if $W \neq W'$.
Hence, $e_3(\A) \geq 2$, by  equality \eqref{eq:bmult}, and we are done.

Suppose then that $W=W'$. We know from \cite{FY} that actually $W=\phi ^*\mathbb{T}(S)$, since $\phi$ comes from a $3$-net.
Taking tangent spaces at the origin $1 \in\mathbb{T}(M_\A)$, we infer that $\phi'^* H^1(S, \C) \subseteq \phi^* H^1(S, \C)$.

We identify $H^1(M_\A, \C)$ with $\C^{\A}$, using the distinguished $\Z$-basis. 
In this way, the subspace $\phi^* H^1(S, \C)$ (respectively $\phi'^* H^1(S, \C)$) is identified with 
the subset of those elements $\eta \in \C^{\A}$ (respectively $\eta' \in \C^{\A}$)
taking the constant values $a,b,c$ (respectively $a', b', c'$) on the blocks of $\Pi$ (respectively $\Pi'$), 
where $a+b+c=a'+b'+c'=0$.
Now, it is an easy matter to check that $\phi^* H^1(S, \C) \cap \phi'^* H^1(S, \C)=0$.
This contradiction finishes the proof of Theorem \ref{thm:a}. 
\hfill $\Box$

\subsection{Full monodromy action}
\label{ss46}

It follows from decomposition \eqref{eq:monintro} that the characteristic polynomial 
$\Delta_\A(t)=(t-1)^{|\A|-1} \Pi_{1< d|n} \Phi _d(t)^{e_d(\A)}$ 
encodes the full monodromy action on $H_1(F_\A, \Q)$.

The approach via modular bounds works only for prime power monodromy multiplicities, $e_{p^s}(\A)$.
One way to avoid this inconvenience is to impose restrictions on multiplicities of $2$-flats,
like in \cite{PS14} for instance, to arrive at full monodromy computations.
Unfortunately, as we saw in \S\ref{ss23}, arbitrarily high flat multiplicities
may appear for non-exceptional complex reflection arrangements.

Even in this kind of situation, there is hope related to the following well-known vanishing  criterion (see e.g. \cite{MP}):
if $d \notdivides |\A_X|$ for any $X \in\LL_2(\A)$, then $e_d(\A)=0$. 
It turns out that this works for full monomial arrangements of small rank.
The result below verifies in particular the strong form of the conjecture from \cite{PS14}.

\begin{prop}
\label{prop:full}
For $\A=\A(m,1,l)$, with $l=3$ or $4$, $\Delta_{\A}(t)=(t-1)^{|\A|-1}(t^2+t+1)$, 
if $l=3$ and $m \equiv 1 (\text{mod} \; 3)$,
and $\Delta_{\A}(t)=(t-1)^{|\A|-1}$, otherwise.
\end{prop}

\begin{proof}
We have to compute $e_d(\A)$ for all divisors $d>1$ of $|\A|$. If $d$ is prime, 
this was done in Theorem \ref{thm:a} and Theorem \ref{thm:b}.
It follows from \S\ref{ss23} that the $2$-flats of $\A$ have multiplicities
$2,3$ or $m+2, |\A| = 3(m+1)$ for $l=3$ and $|\A| = 2(3m+2)$ for $l=4$.
If $d$ is not prime and $e_d(\A) \neq 0$, the vanishing criterion forces $m \equiv -2 \; (\bmod\ d)$.
Writing that $|\A| \equiv 0 \;(\bmod \; d)$, we obtain for $l=3$ that $3 \equiv 0 \;(\bmod\;  d)$,
a contradiction, and $8 \equiv 0 \;(\bmod\;  d)$, for $l=4$.
In the second case, the modular bound implies that $\beta_2(\A)>0$, contradicting Theorem \ref{thm:b}(\ref{b2}).
\end{proof}

\begin{ack}
We are grateful to Alex Suciu for helpful discussions.
The second author thanks Universit\' e de Nice Sophia Antipolis, where
this work was completed, for support and hospitality. 
\end{ack}

\newcommand{\arxiv}[1]
{\texttt{\href{http://arxiv.org/abs/#1}{arxiv:#1}}}
\newcommand{\arx}[1]
{\texttt{\href{http://arxiv.org/abs/#1}{arXiv:}}
\texttt{\href{http://arxiv.org/abs/#1}{#1}}}

\end{document}